\title{
Bipartite decomposition of random graphs
}
\author{Noga Alon
\thanks{Sackler School of Mathematics
and Blavatnik School of Computer Science,
Tel Aviv University,
Tel Aviv 69978, Israel.
Email: {\tt nogaa@tau.ac.il}.
Research supported in part by an ISF grant, by
a USA-Israeli BSF grant and by 
the Israeli I-Core program.
}
}

\documentclass[11pt]{article}
\usepackage{amsfonts,amssymb,amsmath,latexsym}
\newenvironment{proof}
      {\medskip\noindent{\bf Proof.}\hspace{1mm}}
      {\hfill$\Box$\medskip}

\def\qed{\ifvmode\mbox{ }\else\unskip\fi\hskip 1em plus 10fill$\Box$}

\newtheorem{theo}{Theorem}[section]
\newtheorem{prop}[theo]{Proposition}
\newtheorem{lemma}[theo]{Lemma}
\newtheorem{coro}[theo]{Corollary}
\newtheorem{conj}[theo]{Conjecture}

\newcommand{\FF}{{\cal F}}
\newcommand{\MM}{{\cal M}}

\begin{document}
\date{}

\maketitle

\begin{abstract}
For a graph $G=(V,E)$, let $\tau(G)$ denote the minimum number of
pairwise edge disjoint 
complete bipartite subgraphs of $G$ so that each edge of $G$
belongs to exactly one of them. It is easy to see that for every
graph $G$, 
$\tau(G) \leq n -\alpha(G)$, where $\alpha(G)$ is the maximum size
of an independent set of $G$. Erd\H{o}s conjectured in the 80s
that for almost every graph $G$ equality holds,
i.e., that for the random graph $G(n,0.5)$,
$\tau(G)=n-\alpha(G)$ with high probability, 
that is, with probability that tends to $1$ as $n$
tends to infinity. Here we show that this conjecture is (slightly)
false, proving that for most values of $n$ tending to infinity 
and for $G=G(n,0.5)$,
$\tau(G) \leq n-\alpha(G)-1$ with high probability, and that for
some sequences of values of $n$ tending to infinity
$\tau(G) \leq n-\alpha(G)-2$ with probability bounded away from $0$.
We also study the typical value of $\tau(G)$ for random
graphs $G=G(n,p)$ with $p < 0.5$ and show that there is an absolute
positive constant  $c$ so that for all $p \leq c$ and for
$G=G(n,p)$, $\tau(G)=n-\Theta(\alpha(G))$ with high probability.
\end{abstract}
\section{Introduction}
For a graph $G=(V,E)$, let $\tau(G)$ denote the minimum number of
pairwise edge disjoint 
complete bipartite subgraphs of $G$ so that each edge of $G$
belongs to exactly one of them. A well known theorem of Graham and
Pollak \cite{GP} asserts that $\tau(K_n)=n-1$, see \cite{Tv},
\cite{Pe}, \cite{Vi} for more proofs, and \cite{Al}, \cite{KRW} for
several variants.

Let $\alpha(G)$ denote the maximum size of an independent set of
$G$. It is easy to see that for every graph $G$,
$\tau(G) \leq n -\alpha(G)$. Indeed one can partition all edges
of $G$ into $n-\alpha(G)$ stars centered at the vertices of the
complement of a maximum independent set in $G$.
Erd\H{o}s conjectured (see
\cite{KRW}) that for almost every graph $G$ equality holds,
i.e., that for the random graph $G(n,0.5)$,
$\tau(G)=n-\alpha(G)$ with high probability 
({\em whp}, for short),
that is, with probability that tends to $1$ as $n$
tends to infinity.

Chung and Peng \cite{CP} extended the conjecture for the random
graphs $G(n,p)$ with $p \leq 0.5$, conjecturing that for any
$p \leq 0.5,~~ $
$\tau(G) =n -(1+o(1)) \alpha(G)$ whp. They also established lower
bounds supporting this conjecture, and the one of Erd\H{o}s,
by proving that for $G=G(n,p)$ and
for all $0.5 \geq p
\geq \Omega(1)$, $\tau(G) \geq n-o((\log n)^{3+\epsilon})$ for
any positive  $\epsilon$,
and that 
for $p=o(1)$ and $p = w(\frac{\log^2 n}{\sqrt n})$,
$\tau(G) \geq n-o((\frac{\log^3 n}{p^2})^{1+\eta})$ whp for any 
positive $\eta$.

Here we first show that Erd\H{o}s' conjecture for $G=G(n,0.5)$
is (slightly) incorrect. It turns out that for most values of $n$,
and for $G=G(n,0.5),~~~$
$\tau(G) \leq n - \alpha(G)-1$ whp, while for some exceptional 
values of $n$ (that is, those values for which the size of 
$\alpha(G)$ is concentrated in two points, and not in one),
$\tau(G) \leq n-\alpha(G)-2$ with probability that is bounded  away
from $0$. As far as we know it may be possible that for these values 
of $n~~$
$\tau(G) =n - \alpha(G)$ with probability bounded away from $0$
(but not with  probability that tends to $1$ as $n$ grows).

To state the result precisely let $\beta(G)$ denote the largest number
of vertices in an induced complete bipartite subgraph of $G$.
It is easy to see that for every $G$, $\tau(G) \leq n-\beta(G)+1$.
Indeed, one can decompose all edges of $G$ into $n-\beta(G)$ stars
centered at the vertices of the complement of an induced complete
bipartite subgraph $H$ of $G$ of maximum size, together with $H$
itself.
For an integer $n$ let $k_0=k_0(n)$ denote the largest integer $k$
so that  $f(k)={n \choose k}2^{-{k \choose 2}} \geq 1$. In words,
$k_0$ is the largest $k$ so that the expected number of independent
sets of size $k$ in $G=G(n,0.5)$ is at least $1$.
It is easy to check that $k_0=k_0(n)=(1+o(1)) 2 \log_2 n$,
that 
$n=\Theta( k_0 2^{k_0/2}) $ 
and and that for 
$k=(1+o(1))k_0$, $f(k+1)/f(k) =n^{-1+o(1)}$, (c.f., e.g.,
\cite{AS}). 
\begin{theo}
\label{t11}
Let $k_0=k_0(n)$ be as above. Then
\vspace{0.2cm}

\noindent
(i) If $1=o(f(k_0))$ and $f(k_0+1)=o(1)$ then whp
$\alpha(G)=k_0$ and $\beta(G)=k_0+2$. Therefore, in this case
$\tau(G) \leq n-\alpha(G)-1$ whp.
\vspace{0.2cm}

\noindent
(ii) If $f(k_0)=\Theta(1)$ then whp one of of the following four
possibilities holds, and each of them holds 
with probability that is bounded away from $0$ and $1$:

\noindent
(a) $\alpha(G)=k_0$ and $\beta(G)=k_0+2$.

\noindent
(b) $\alpha(G)=k_0$ and $\beta(G)=k_0+1$.

\noindent
(c) $\alpha(G)=k_0-1$ and $\beta(G)=k_0+2$.

\noindent
(d) $\alpha(G)=k_0-1$ and $\beta(G)=k_0+1$.
\vspace{0.2cm}

\noindent
(iii) If $f(k_0+1)=\Theta(1)$ then each of the four possibilities obtained 
from the ones above by replacing $k_0$ by $k_0+1$
is obtained with
probability bounded away from $0$ and $1$, and whp one of those holds.
\end{theo}
We also improve the estimates of \cite{CP} for $G(n,p)$ for any
$c \geq p \geq \frac{2}{n}$, 
where $c$ is some small positive absolute constant,
determining the typical value of $n-\tau(G(n,p))$ up to a constant
factor in all this range.
\begin{theo}
\label{t12}
There exists an absolute constant $c>0$ so that for any 
$p$ satisfying $\frac{2}{n} \leq p \leq c$ and for $G=G(n,p)$
$$
\tau(G)= n-\Theta(\frac{\log (np)}{p})
$$
whp.
\end{theo}

For
very sparse graphs, that is, for $p = o(n^{-7/8})$, it is not
difficult to give a precise expression for the typical value of
$\tau(G)$. For a graph $H$ in which every connected component is
either an isolated vertex or a cycle of length $4$, let 
$\gamma(H)$ denote the number of vertices of $H$ minus the number
of cycles of length $4$ in it. 
\begin{prop}
\label{p13}
If $p=o(n^{-7/8})$ then for $G=G(n,p)$, whp, $\tau(G)=n-max
(\gamma(H))$, where the maximum is taken over all induced subgraphs
of $G$ in which any connected component is either a vertex or a
cycle of length $4$.
\end{prop}

The rest of this paper  contains the proofs.
Theorem \ref{t11} is proved in Section 2. Part (i) 
is established using the second moment method and
parts (ii) and (iii)  are proved by applying the 
Stein-Chen method. 

Theorem \ref{t12} is proved in Section 3 by combining an appropriate first
moment computation
with some combinatorial arguments. Section 4 contains several concluding
remarks as well as the
simple  proof of Proposition  \ref{p13}.

Throughout the rest of the paper we assume, whenever this is needed, that
$n$ is sufficiently large. 
All logarithms are in base $2$, unless otherwise specified.
\section{Random graphs}\,

In this section we consider $G=G(n,0.5)$ and prove Theorem \ref{t11}. 

We start with the proof
of part (i), which implies that for most values of $n$,
$\tau(G) \leq n-\alpha(G)-1$. Here "most" means that if we take a
random uniform integer $n$ in $[1,M]$, then the probability that
for this $n$ the assumptions in part (i) hold tend to $1$ as $M$
tends to infinity.

The proof of part (i) is based on the second moment method. Let
$V=\{1,2, \ldots ,n\}$ be a fixed set of $n$ labeled vertices,
and let $G=G(n,0.5)=(V,E)$ be the  random graph on $V$. Let 
$f(k)={n \choose k}2^{-{k \choose 2}}$ be
the expected number of independent sets of
size $k$ in $G$, and let $k_0$ be, as in the introduction, the
largest $k$ so that $f(k) \geq 1$. Suppose that the assumption in
Theorem \ref{t11}, part (i)  holds. This means that the expected
number of independent sets of size $k_0+1$ in $G$ is $o(1)$ and
hence, by Markov's Inequality, the probability that there is such
an independent set if $o(1)$. The assumption also implies that the
expected number of independent sets of size $k_0$ tends to
infinity. It is known (c.f., e.g., \cite{AS}, Theorem 4.5.1)
that in this case $\alpha(G)=k_0$ whp. For completeness we include
the relevant computation, which will be used later as well.

Suppose  $k=(1+o(1))2 \log_2 n$. 
For each $K \subset V$, $|K|=k$, let $X_K$ be the indicator random
variable whose value is $1$ iff $K$ is an independent set in $G$.
Let $X=\sum_K X_K$, where $K$ ranges over all subsets of size $k$
of $V$,  be the total number of independent sets of size
$k$ in $G$. The expectation of this random variable is  clearly
$E(X)=f(k)={n \choose k}2^{-{k \choose 2}}$. We proceed to estimate
its variance. For $K,K' \subset V$, $|K|=|K'|=k$, let $K \sim K'$
denote that $|K \cap K'| \geq 2$ (and $K \neq K'$). The variance
of $X$ satisfies:
$$
\mbox{Var}(X)=\sum_K \mbox{Var}(X_K)+\sum_{K \sim K'}
\mbox{Cov}(X_K,X_{K'}) \leq E(X) +\sum_{K \sim K'} E(X_K X_{K'}),
$$
where $K,K'$ range over all ordered pairs of subsets of size $k$ of
$V$ satisfying $2 \leq |K \cap K'| \leq k-1$.
Note that
$$
\sum_{K \sim K'} E(X_K X_{K'}) 
=\sum_{i=2}^{k-1} {n \choose k} { k \choose i}
{{n-k} \choose {k-i}}2^{-2{k \choose 2} + {i \choose 2}}= 
\sum_{i=2}^{k-1} f_i,
$$
where here
$$
f_i =
{n \choose k} { k \choose i}
{{n-k} \choose {k-i}}2^{-2{k \choose 2} + {i \choose 2}}
$$
is the contribution to the sum $\sum_{K \sim K'} E(X_K X_{K'})$
arising from ordered pairs $K,K'$ whose intersection is of size
$i$.

Without trying to get here the best possible estimate, we consider
two possible ranges for the parameter $i$, as follows.
\vspace{0.2cm}

\noindent
{\bf Case 1:}\, If $2 \leq i \leq 2k/3$ then
$$
\frac{f_i}{f(k)^2} = \frac{{k \choose i} {{n-k} \choose {k-i}}}{{n
\choose k}}  2^{{i \choose 2}} \leq k^i (\frac{k}{n})^i 2^{{i
\choose 2}} =(\frac{k^2 2^{(i-1)/2}}{n})^i \leq \frac{1}{n^{0.3i}}.
$$
Here we used the facts that $k=(1+o(1)) 2 \log_2 n$ and
$i \leq 2k/3$ to conclude that 
$$
\frac{k^2 2^{(i-1)/2}}{n} \leq \frac{1}{n^{1/3-o(1)}}.
$$
\vspace{0.2cm}

\noindent
{\bf Case 2:}\, If $i=k-j,~ 1 \leq j \leq k/3$, then
$$
\frac{f_i}{f(k)} = {k \choose j} {{n-k} \choose {j}}
 2^{-{k \choose 2}+{i \choose 2}} \leq k^j n^j 2^{-j(k-j)}
\leq (kn2^{-(k-j)})^j \leq \frac{1}{n^{0.3j}}.
$$
We have thus proved the following.
\begin{lemma}
\label{l21}
With the notation above, if $k=(1+o(1)) 2 \log_2 n$ and $i \leq
2k/3$,
then $f_i \leq f(k)^2 \frac{1}{n^{0.3i}}$. If 
$k=(1+o(1)) 2 \log_2 n$ and $i =k-j, ~j \leq k/3$, then
$f_i \leq f(k) \frac{1}{n^{0.3j}}.$ Therefore, if $f(k) \geq
\Omega(1)$ then 
$\sum_{K \sim K'} E(X_KX_{K'}) =o(f(k)^2)$ and 
$\mbox{Var}(X) \leq E(X)
+o(f(k)^2)=E(X)+o((E(X)^2).$
\end{lemma}

Next we consider induced complete bipartite graphs in the random
graph $G=G(n,0.5)$ on $V$. Let $k=(1+o(1))2 \log_2 n$ satisfy
$n=\Theta(k 2^{k/2})$ and recall that this holds for $k=k_0(n)$ 
defined as the largest integer $k$ so that $f(k) \geq 1$. For any
subset $B \subset V$ of size $|B| =k+2$ let $Y_B$ denote the
indicator random variable whose value is $1$ iff the induced
subgraph of $G$ on $B$ is a complete bipartite graph. Define
$Y=\sum_B Y_B$, as $B$ ranges over all subsets of size $k+2$ of
$G$, and note that this is the number of induced complete bipartite
subgraphs of $G$ of size $k+2$. Denote the expected value of $Y$ by
$g(k)$ and note that
$$
E(Y)=g(k)={n \choose {k+2}}(2^{k+1}-1)2^{-{{k+2} \choose 2}}.
$$
Indeed, there are ${n \choose {k+2}}$ subsets $B$ of $k+2$ vertices,
in each such subset there are $2^{k+1}-1$ ways to partition it into
two nonempty vertex classes, and the probability that the induced
subgraph on $B$ is a complete bipartite graph on these two vertex
classes is $2^{-{{k+2} \choose 2}}$.

Since by assumption $n=\Theta(k 2^{k/2})$ it follows that
\begin{equation}
\label{e21}
g(k)=f(k) \frac{(n-k)(n-k-1)}{(k+2)(k+1)} (2^{k+1}-1)
2^{-2k-1}=\Theta(f(k)).
\end{equation}

To compute the variance of $Y$ let $B \sim B'$ denote,
for two subsets
$B,B' \subset V$, each of cardinality $k+2$ , that $2 \leq |B \cap
B'| $ and $B \neq B'$. Then
$$
\mbox{Var}(Y) \leq E(Y) + \sum_{B \sim B'} \mbox{Cov}(Y_B,Y_{B'})
\leq E(Y) + \sum_{B \sim B'} E(Y_BY_{B'}).
$$
Now,
$$
\sum_{B \sim B'} E(Y_B Y_{B'}) \leq \sum_{i=2}^{k+1} {n \choose
{k+2}} (2^{k+1}-1)2^{-{{k+2} \choose 2}} {{k+2} \choose i}
{{n-k-2} \choose {k+2-i}} 2^{k+2-i} 2^{-{{k+2} \choose 2} +{i
\choose 2}}=\sum_{i=2}^{k+1} g_i,
$$
where 
$$
g_i=
{n \choose
{k+2}} (2^{k+1}-1)2^{-{{k+2} \choose 2}} {{k+2} \choose i}
{{n-k-2} \choose {k+2-i}} 2^{k+2-i} 2^{-{{k+2} \choose 2} +{i
\choose 2}}
$$
is the contribution from pairs $B,B'$ with intersection of size
$i$.

We bound the terms $g_i$ as done for the
quantities $f_i$ before.
\vspace{0.2cm}

\noindent
{\bf Case 1:}\, If $2 \leq i \leq 2k/3+2$ then
$$
\frac{g_i}{g(k)^2} <(k+2)^i (\frac{k+2}{n})^i 2^{{i \choose 2}}
=(\frac{(k+2)^2 2^{(i-1)/2}}{n})^i 
\leq \frac{1}{n^{0.3i}}.
$$
\vspace{0.2cm}

\noindent
{\bf Case 2:}\, If $i=k+2-j,~1 \leq j \leq k/3$, then
$$
\frac{g_i}{g(k)} = {{k+2} \choose j} {{n-k-2} \choose {j}}2^j
 2^{-{{k+2} \choose 2}+{i \choose 2}} \leq 
((k+2)2n2^{-(k+2-j)})^j \leq \frac{1}{n^{0.3j}}.
$$
We have thus obtained the following.
\begin{lemma}
\label{l22}
With the notation above, if $
n=\Theta(k 2^{k/2})$ and $i \leq
2k/3+2$,
then $g_i \leq g(k)^2 \frac{1}{n^{0.3i}}$. If 
$n=\Theta(k 2^{k/2})$ and $i =k+2-j, ~j \leq k/3$, then
$g_i \leq g(k) \frac{1}{n^{0.3j}}.$ Therefore, if $g(k) \geq
\Omega(1)$ then 
$\sum_{B \sim B'} E(Y_B Y_{B'}) =o(g(k)^2)$ and 
$\mbox{Var}(Y) \leq E(Y)
+o(g(k)^2)=E(Y)+o((E(Y)^2).$
\end{lemma}
\vspace{0.2cm}

\noindent
{\bf Proof of Theorem \ref{t11}, part (i):}\,
Since $f(k_0+1)=o(1)$ the expected number of independent sets of
size $k_0+1$ is $o(1)$ and hence, by Markov, with probability
$1-o(1)$,  $\alpha(G) < k_0+1$. On the other hand, as $f(k_0)$ tends 
to infinity we conclude, by Lemma \ref{l21}, that the random
variable $X$ which counts the number of independent sets of size
$k_0$  in $G$ has expectation $f(k_0)$ which tends to infinity, and
variance $o(f(k_0)^2)$. Thus, by Chebyshev's Inequality, $X$ is
positive whp, and therefore $\alpha(G) \geq k_0$ (and hence 
$\alpha(G)=k_0$) whp.

The situation with $Y$ is similar. By (\ref{e21})
$g(k_0)=\Theta(f(k_0))$ and
$g(k_0+1)=\Theta(f(k_0+1))$. Therefore, by assumption,
$g(k_0+1)=o(1)$ and hence $\beta(G)<(k_0+1)+2=k_0+3$ whp.  
On the other hand,
by Lemma \ref{l22}, and since by assumption $g(k_0)=\Theta(f(k_0))$
tends to infinity, we conclude, by Chebyshev's Inequality, that 
$\beta(G) \geq k_0+2$ whp. Thus $\beta(G)=k_0+2$ whp, implying
the assertion of part (i).
\vspace{0.2cm}

We proceed with the proof of part (ii) (the proof of part (iii) is
essentially identical). This is done by applying the Stein-Chen
method, which is a method that can show that certain random
variables can be approximated well by Poisson random variables. 
It is in fact possible to apply the two-dimensional method
(see, for example, \cite{BHJ}, Corollary 10.J.1) to show that if
$f(k_0)=\Theta(1)$ (and hence also $g(k_0)=\Theta(1))$, then the
two random variables $X$, which counts the number of independent
sets of size $k=k_0$, and $Y$, which counts the number of induced
complete bipartite subgraphs of size $k_0+2$, behave
approximately like independent Poisson random variables with 
expectations $E(X)$ and $E(Y)$. In particular, each of the four
events 
\begin{equation}
\label{e22}
E_{11}=\{X>0,Y>0\},~ E_{10}=\{X>0,Y=0\},~ E_{01}=\{X=0,
Y>0\}~ E_{00}=\{X=0,Y=0\}
\end{equation}
are obtained with probability bounded away from $0$ and $1$.
However, the same conclusion can be derived using the one
dimensional method, since it suffices to show that $X$, $Y$ and
their sum $X+Y$ are all approximately Poisson.  This suffices
to show that if $E(X)=\lambda$ and $E(Y)=\mu$, then
the probability that $X=0$ is $(1+o(1))e^{-\lambda}$, the
probability that $Y=0$ is $(1+o(1))e^{-\mu}$ and the probability
that $X+Y=0$ (which is exactly the probability that $X=Y=0$, as
both are nonnegative integers) is $(1+o(1))e^{-\lambda-\mu}$. This will
enable one to compute the probabilities of all four events 
$E_{ij}$ in (\ref{e22}) above and establish the conclusion of
Theorem \ref{t11}, part (ii).

The details follow. We start with a statement of the Stein-Chen
method in a simple form that suffices for our purpose here. This
is the version that appears in \cite{JLR}, Theorem 6.23.

Let $\{I_{\alpha}\}_{\alpha \in \FF}$ be a (finite) family of 
indicator random variables. A graph $L$ 
on the set of vertices $\FF$ 
is a dependency graph for this family if 
for any two disjoint subsets $\MM_1$ and $\MM_2$ of $\FF$ with no
edges of $L$ between them, the families 
$\{I_{\alpha}\}_{\alpha \in \MM_1}$ and
$\{I_{\beta}\}_{\beta \in \MM_2}$ are mutually independent.
Thus, for example, if the family of indicator random variables
is the family of all ${n \choose k}$ variables $X_K$ considered in
the paragraphs preceding Lemma \ref{l21}, then the graph $L$ in
which $K,K'$ are adjacent iff $K \sim K'$, that is, iff
$2 \leq |K \cap K'| \leq k-1$, is a dependency graph. We need the
following version of the Stein-Chen method.
\begin{theo}[c.f., \cite{JLR}, Theorem 6.23]
\label{t23}
Let $\{I_{\alpha}\}_{\alpha \in \FF}$ be a (finite) family of 
indicator random variables with dependency graph $L$.
Put $X=\sum_{\alpha \in \FF} X_{\alpha}$, let $\pi_{\alpha}$ be the
expectation of $I_{\alpha}$ and let $\lambda=\sum_{\alpha \in \FF}
\pi_{\alpha}$ be the expectation of $X$. Then the total variation
distance between the distribution of $X$ and that of a Poisson
random variable $Po(\lambda)$ with expectation $\lambda$ satisfies
$$
d_{TV}(X, Po(\lambda)) \leq \mbox{min}(\lambda^{-1},1)
(\sum_{\alpha \in \FF} \pi_{\alpha}^2 +
\sum_{\alpha,\beta \in \FF, \alpha \beta \in E(L)} 
(E(I_{\alpha}I_{\beta})+E(I_\alpha)E(I_{\beta}) )),
$$
where the sum is over ordered pairs $\alpha,\beta$.
In particular, $|\mbox{Prob}(X=0)-e^{-\lambda}|$ is bounded by the
right hand side of the last inequality.
\end{theo}

We can now proceed with the proof of Theorem \ref{t11}, part (ii).
Let $G=G(n,1/2)$ be the random graph on $V=\{1,2,\ldots ,n\}$, let
$k_0$ be as in Theorem \ref{t11}, and
suppose that the assumption of part (ii) holds, that is
$f(k_0)=\Theta(1)$. Let $X=\sum_{K} X_K$ be, as before, the number 
of independent  sets of size $k=k_0$ in $G$, then $E(X)=f(k_0)$.
Put $f(k_0)=\lambda$. As noted before, the graph on
the $k$-subsets $K$ of $V$ in which $K,K'$ are adjacent iff 
$K \sim K'$ is a
dependency graph for the variables $X_K$. Put $\pi_K=E(X_K)=
2^{-{k \choose 2}}$.
By Theorem \ref{t23}:
\begin{equation}
\label{e23}
|\mbox{Prob}(X=0)-e^{-\lambda}| 
\leq 
\mbox{min}(\lambda^{-1},1)
(\sum_{K} \pi_{K}^2 +
\sum_{K \sim K'} 
(E(X_K X_{K'})+E(X_K)E(X_{K'}) )),
\end{equation}
where the first sum is over all $k$-subsets $K$ of $V$ and the
second is over ordered pairs of such subsets that satisfy
$K \sim K'$. 

Since $\pi_K =2^{-{k \choose 2}} =n^{-\Theta(\log n)}
=o(1)$,  it follows that
$$
\sum_{K} \pi_{K}^2=2^{-{k \choose 2}} \sum_{K} \pi_K=o(1) \lambda
=o(1).
$$
It is also easy to bound the sum
$$
\sum_{K \sim K'} E(X_K)E(X_{K'}) 
$$ 
as the fraction  of pairs $K,K'$ that satisfy $K \sim K'$ among all
pairs $K,K'$ is easily seen to be $\Theta(k^4/n^2)=o(1)$. Therefore
$$
\sum_{K \sim K'} E(X_K)E(X_{K'})=O(k^4/n^2) (\sum_K \pi_K)^2
=o(1) \lambda^2=o(1).
$$
It remains to bound the sum
$$
\sum_{K \sim K'} 
E(X_K X_{K'}).
$$
By Lemma \ref{l21} this is at most $o(\lambda^2)=o(1)$.

Plugging in (\ref{e23}) we conclude that
\begin{equation}
\label{e24}
\mbox{Prob}(X=0)=(1+o(1))e^{-\lambda}.
\end{equation}

A similar computation shows that for the random variable $Y$ that
counts the number of induced complete bipartite subgraphs of size
$k+2=k_0+2$ in $G$, whose expectation is $g(k_0)=\Theta(f(k_0))=
\Theta(1)$, which we denote by $\mu=g(k_0)$, we have
\begin{equation}
\label{e25}
\mbox{Prob}(Y=0)=(1+o(1))e^{-\mu}.
\end{equation}

Indeed, here $Y=\sum_{B} Y_B$ where $B$ ranges over all subsets of
cardinality $k+2$ of $V$ and $Y_B$ is the indicator random variable 
whose value is $1$ iff the induced subgraph on $B$ is a complete
bipartite graph.  A dependency graph here is obtained by having
$B,B'$ adjacent iff $B \sim B'$, that is, iff
$2 \leq |B \cap B'| \leq k+1$. One can thus apply Theorem \ref{t23}
and establish (\ref{e25}) by repeating the arguments in the proof
of (\ref{e24}), replacing Lemma \ref{l21} by Lemma \ref{l22}. 

Finally, we claim that the sum $X+Y$ can also be approximated well
by a Poisson random variable  with expectation $\lambda+\mu$ and
hence 
\begin{equation}
\label{e26}
\mbox{Prob}(X=Y=0)=
\mbox{Prob}(X+Y=0)=
(1+o(1))e^{-\lambda-\mu}.
\end{equation}

The reasoning here is similar, although it requires a slightly more
tedious computation. Here $X+Y=\sum_K X_k + \sum_B Y_B$ with
$X_K,Y_B$ as before. A dependency graph $L$ is obtained here by having
$K,K'$ adjacent iff $K \sim K'$, $B,B'$ adjacent iff
$B \sim B'$, and $K,B$ adjacent iff $2 \leq |K \cap B| \leq k$
(note that here the subset $B$ may fully contain the subset $K$).
Here $E(X_K)=\pi_K=o(1)$ and $E(Y_B)=\pi_B=o(1)$ and hence, as
before
$$
\sum_K \pi_K^2+\sum_B \pi_B^2 =o(1)(\lambda+\mu)=o(1).
$$
As before
$$
\sum_{KK' \in E(L)} 
E(X_K)E(X_{K'})=O(k^4/n^2) (\sum_K \pi_K)^2
=o(1) \lambda^2=o(1),
$$
and similarly
$$
\sum_{BB' \in E(L)} 
E(Y_B)E(Y_{B'})=O(k^4/n^2) (\sum_B \pi_B)^2
=o(1) \mu^2=o(1)
$$
and
$$
\sum_{KB \in E(L)} 
E(X_K)E(Y_{B})=O(k^4/n^2) (\sum_{K} \pi_K)(\sum_B \pi_B)
=o(1) \lambda\mu=o(1).
$$
The remaining term we have to bound, which is also the main term,
is
$$
\sum_{KK' \in E(L)} E(X_K X_{K'})+
\sum_{BB' \in E(L)} E(Y_B Y_{B'})+
\sum_{KB \in E(L)} E(X_K Y_{B}).
$$
Each of the first two summands here is $o(1)$, by the discussion
above. The third sum can be bounded by a similar computation, which
follows.
$$
\sum_{KB \in E(L)} E(X_K Y_{B})
=\sum_{i=2}^k {n \choose k}2^{-{k \choose 2}} {k \choose i}
{{n-k} \choose {k+2-i}}(2^{k+2-i}-1) 2^{-{{k+2} \choose 2}+{i
\choose 2}} =\sum_{i=2}^k h_i,
$$
where here
$$
h_i=
{n \choose k}{k \choose i} {{n-k} \choose {k+2-i}}
(2^{k+2-i}-1) 2^{-{k \choose 2}-{{k+2} \choose 2}+{i \choose 2}} 
$$
is the contribution arising from pairs $K,B$ with $|K \cap B|=i$.
Indeed, there are ${n \choose k}$ ways to choose $K$, then 
${k \choose i}$ ways to choose the intersection $K \cap B$ and
${{n-k} \choose {k+2-i}}$ to select the remaining elements of $B$.
Next we have to choose for each of these remaining elements if it
belongs to the same vertex class of the induced bipartite graph on
$B$ as the elements of $K \cap B$, or to the other vertex class
(and not all elements can belong to the same vertex class as those
of $K \cap B$, since otherwise we get an independent set and not 
a complete bipartite graph). There are $2^{k+2-i}-1$ ways to make
this choice. Finally, the ${k \choose 2}+{{k+2} \choose 2}-{i
\choose 2}$ edges of $K$ and $B$ should all be as needed, and the
probability for this is $2^{-{k \choose 2}-{{k+2} \choose 2}+{i
\choose 2}}$.

To bound $h_i$ we consider two possible ranges of the parameter
$i$, as done in the proofs of Lemmas \ref{l21} and \ref{l22}.
\vspace{0.2cm}

\noindent
{\bf Case 1:}\, If $2 \leq i \leq 2k/3$ then, since
$n=\Theta(k2^{k/2})$, 
$$
\frac{h_i}{f(k)^2} = \frac{{k \choose i} {{n-k} \choose {k+2-i}}}{{n
\choose k}}  2^{{i \choose 2}} 2^{-2k-1} (2^{k+2-i}-1)
\leq k^i (\frac{k}{n})^{i-2} 2^{-k} 2^{{i \choose 2}} 
$$
$$
=\Theta(k^i \frac{k^{i-2}}{n^{i-2}} (\frac{k}{n})^2 2^{{i \choose
2}})
=\Theta( (\frac{k^2 2^{(i-1)/2}}{n})^i )
\leq \frac{1}{n^{0.3i}}.
$$
\vspace{0.2cm}

\noindent
{\bf Case 2:}\, If $i=k-j,~0 \leq j \leq k/3$, then
$$
\frac{h_i}{f(k)} \leq {k \choose j} {{n-k} \choose {j+2}}
 2^{-{{k+2} \choose 2}+{i \choose 2}} 2^{j+2}
\leq k^j n^j 2^{-(j+2)(k-j)} 2^{j+2}
\leq (kn2^{-(k-j)}2)^{j+2} \leq \frac{1}{n^{0.3(j+2)}}.
$$
By the bounds above and Theorem \ref{t23}, (\ref{e26}) follows.
\vspace{0.2cm}

\noindent
{\bf Proof of Theorem \ref{t11}, parts (ii), (iii):}\,
Suppose the  assumptions of part (ii) hold. Then the expected
number of independent sets of size $k_0$ is $\lambda=\Theta(1)$,
and the expected number of induced complete bipartite graphs of
size $k_0+2$ is $\mu=\Theta(1)$. Note that this implies that the
expected number of independent sets of size $k_0-1$ is $n^{1-o(1)}$
and hence there are such sets whp, by Lemma \ref{l21}, and the 
expected number of independent sets of size $k_0+1$ is
$n^{-1+o(1)}$, and hence, by Markov's Inequality, whp there are no
such sets. Thus $\alpha(G)$ is either $k_0-1$ or $k_0$ whp.
Similarly, by Lemma \ref{l22}, $\beta(G)$ is either 
$k_0+1$ or $k_0+2$ whp.

Let $X,Y$ be the
random variables as above. Then by (\ref{e24}),(\ref{e25}) and (\ref{e26})
each of the four events $E_{ij}$ in (\ref{e22}) occurs with
probability bounded away from $0$ and $1$ (which we can compute, up
to a $(1+o(1))$ factor, as a function of $\lambda$ and $\mu$ which
are both $\Theta(1)$.) Also, by the previous paragraph, whp
exactly one of these events holds.

Note, now, that if $E_{11}$ holds then there is an independent set
of size $k_0$ and there is an induced complete bipartite graph
of size $k_0+2$, namely, in this case the assertion of Theorem
\ref{t11}, part (ii), (a), holds. Similarly, $E_{10}$ 
corresponds to (b), $E_{01}$ to (c) and $E_{00}$ to (d). This
completes the proof of part (ii). The proof of Part (iii) is
identical, replacing $k_0$ by $k_0+1$.  This completes the proof of
Theorem \ref{t11}. 
\hfill $\Box$
\vspace{0.2cm}

\noindent
{\bf Remark:}\, By the definition
of $k_0$, and as $f(k+1)/f(k)=n^{-1+o(1)}$ for $k$ close to
$k_0$, it follows that $1 \leq f(k_0) \leq n$ and 
$n^{-1+o(1)} \leq f(k_0+1) <1$.
Therefeore, 
for a given $k_0$, exactly one 
of the three possibilities described in parts (i), (ii) and (iii)
of Theorem \ref{t11} occurs. 

\section{Sparser random graphs}\,

In this section we prove Theorem  \ref{t12}. 
We need the following technical lemma.
\begin{lemma}
\label{l31}
There are absolute positive constants $b,c$ and $C$ so that for all
sufficiently large $n$ and every positive $p \leq c$ satisfying 
$np \geq C \log n$ the following holds.
For every integer $m$ satisfying
$$ 
\frac{pn}{16} \leq m \leq \frac{pn}{4}
$$
we have
\begin{equation}
\label{e31}
\sum_{2 \leq d \leq \sqrt m, d | m} {n \choose d} {{n-d} \choose
{m/d}} p^m \leq 2^{-b \log (1/p) m}.
\end{equation}
\end{lemma}
\begin{proof}
Assume, first, that $m$ is even. In this case the sum in
(\ref{e31}) contains the summand ${n \choose 2}{{n-2} \choose
{m/2}}p^m$ which is larger by a factor of $2^{\Omega(m)}
\geq 2^{\Omega(n^{0.5})}$ than each of the other
summands if $m \geq n^{0.5}$, and by a factor of
$n^{\Omega(m)} \geq n^{\Omega(\log n)}$ if $m \leq n^{0.5}$.
Therefore, the left hand side of (\ref{e31}) is
$$
(1+o(1)){n \choose 2} {{n-2} \choose {m/2}}p^m
\leq 2^{(1+o(1))H(\frac{m}{2n})n} p^m=
2^{(1+o(1))H(\frac{m}{2n})n-m \log (1/p)}
$$
where $H(x)=-x \log x -(1-x)\log (1-x)$ is the binary entropy
function, and the $o(1)$ terms tend to zero as $n$ tends to
infinity.

Since for any $x$ smaller than some absolute positive constant 
$H(x)  \leq 1.1 x \log (1/x)$ we conclude that if $c$ is
sufficiently small then for $p \leq c$ and $m$ as above
$$
(1+o(1))H(\frac{m}{2n})n-m \log (1/p)
\leq [1.2 \frac{m}{2n} \log (\frac{2n}{m}) -\frac{m}{2n} 2 \log
(1/p)] n \leq -b' \frac{m}{2n} \log (1/p) n
=-\frac{b'}{2} \log (1/p) m
$$
for some absolute positive constant $b'$, where here we used the
fact that  $\log (\frac{2n}{m})=\log(1/p)+\Theta(1)$ since,
by assumption, $\frac{p}{32} \leq \frac{m}{2n} \leq \frac{p}{8}$.
This supplies the assertion of the lemma in case $m$ is even. 

If $m$ is odd we simply bound the left hand side of (\ref{e31})
by the far bigger quantity ${n \choose {(m+1)/2}} p^{m+1}$,
which is bounded by the right-hand-side of (\ref{e31}), using the
reasoning above. 
\end{proof}

Call a complete bipartite graph nontrivial if it is not a star,
that is, each of its vertex
classes is of size at least $2$.
\begin{lemma}
\label{l32}
There are absolute positive constants $a,c$ and $C$ so that for all
sufficiently large $n$ and every positive  $p \leq c$ satisfying $np
\geq C \log n$, the probability that $G=G(n,p)$ contains a set of at
most $2n$ pairwise edge disjoint nontrivial complete bipartite
graphs whose union covers at least $pn^2/4$ edges is at most
$2^{-a p \log (1/p) n^2}$. 
\end{lemma}
\begin{proof}
If there are such nontrivial complete bipartite subgraphs, omit
each one that contains at most $pn/16$ edges (if there are such
subgraphs). The remaining subgraphs still cover at least
$pn^2/4-2n \cdot pn/16=pn^2/8$ edges. Each such subgraph with
more than  $pn/4$ edges can be partitioned into two nontrivial
complete bipartite subgraphs of nearly equal size, by splitting the
larger vertex class into two nearly equal classes. Repeating this
process we obtain a family of pairwise edge disjoint complete
bipartite subgraphs, each having at least $pn/16$ and at most 
$pn/4$ edges, whose union covers at least $pn^2/8$ edges.  
Let $\FF$  be a family of at most $2n$ arbitrarily chosen members of 
this family, whose union covers at least $pn^2/8$ edges. (If the
whole family contains less than $2n$ subgraphs, let $\FF$ be all of
them, else, take any $2n$ members, since each of them has at least
$pn/16$ edges altogether they cover at least $pn^2/8$ edges).
Put $\FF=\{F_i: i \in I\}$, where $|I| \leq 2n$ and $F_i$ is a
nontrivial complete  bipartite subgraph of $G$ with $m_i$ edges.
Note that by the
discussion above, if $G$ contains a set of at most $2n$ complete
bipartite graphs as in the lemma, then it contains a family $\FF$
as above. 

We complete the proof by establishing an upper bound for the
probability that $G$ contains such a family $\FF$. This is done by 
a simple union bound, using Lemma \ref{l31}. There are $2n$ ways to
choose the size of $I$, then there are less than
$(n^2)^{2n}=n^{4n}$ ways to choose the numbers $m_i$. Once these
are chosen, there are 
$$
\sum_{2 \leq d \leq \sqrt m_i, d | m_i} {n \choose d} {{n-d} \choose
{m_i/d}}
$$ 
ways to select the sets of vertices of the two vertex
classes of $F_i$. As all the graphs $F_i$ are pairwise
edge-disjoint, the probability that all those are indeed
subgraphs of $G$ is at most
$\prod_{i \in I} p^{m_i}$, implying that the probability that there
is an $\FF$ as above is at most
$$
(2n) n^{4n} \prod_{i \in I} 
\sum_{2 \leq d \leq \sqrt m_i, d | m_i} {n \choose d} {{n-d} \choose
{m_i/d}} p^{m_i}.
$$
By Lemma \ref{l31} the last quantity is at most
$$
(2n) n^{4n} \prod_{i \in I} 2^{-b \log (1/p)m_i }
\leq (2n)n^{4n} 2^{-b \log(1/p) pn^2/8} \leq 2^{-a \log(1/p) pn^2}
$$
for some absolute positive constant $a$, where here we used the
fact that $pn \geq C \log n$ which implies that
$(2n)n^{4n} =2^{O(n \log n)} < 2^{o(\log (1/p) pn^2)}.$
This completes the proof. 
\end{proof}

Following Chung and Peng \cite{CP}, 
let $\tau'(G)$ denote the minimum number of
pairwise edge disjoint nontrivial bipartite subgraphs of $G$ whose
union covers all edges of $G$ (if there is no such cover define
$\tau'(G)=\infty$).  Lemma \ref{l32} implies that the probability
that $G=G(n,p)$ for $p$ as in the lemma satisfies $\tau'(G) \leq 2n$
is extremely small, as we observe next.
\begin{coro}
\label{c33}
There are absolute positive constants $a,c$ and $C$ so that for all
sufficiently large $n$ and every positive  $p \leq c$ satisfying $np
\geq C \log n$, the probability that $G=G(n,p)$ 
satisfies $\tau'(G) \leq 2n$ is at most $2^{-apn^2}$.
\end{coro}
\begin{proof}
By the standard estimates for Binomial distributions (c.f., e.g.,
\cite{AS}, Theorem A.1.13) the probability
that $G$ has less than $pn^2/4$ edges is at most
$e^{-(1+o(1))pn^2/16}$. By Lemma \ref{l32}
the probability that $G$ contains a set of at
most $2n$ pairwise edge disjoint nontrivial complete bipartite
graphs whose union covers at least $pn^2/4$ edges is at most
$2^{-a p \log (1/p) n^2}$. If none of these two rare events happens
then clearly $\tau'(G) > 2n$.
\end{proof}

The following lemma 
is proved in \cite{CP} 
\begin{lemma}[\cite{CP}, Lemma 14]
\label{l34}
For any graph $G=(V,E)$ there exists  a set of vertices 
$U \subset V$ so that if $G[U]$ denotes the induced subgraph of $G$
on $U$ then
$$
\tau(G)=|V|-|U|+\tau'(G[U]).
$$
\end{lemma}
The proof is by considering a bipartite decomposition of $G$ into
$\tau=\tau(G)$ complete bipartite subgraphs, with a maximum number
of stars (among all decompositions into $\tau$ such subgraphs). Suppose
that in this decomposition the stars used are centered at the vertices
$V-U$, where $U \subset V$. Now replace each of the remaining,
non-star member $B$ in the decomposition by its induced subgraph on
$V(B) \cap U$. It is easy to see that by modifying the stars, if needed, the
resulting graphs also form a bipartite decomposition  of $G$, and
by the maximality of the number of stars, each of the remaining
subgraphs besides the $|V|-|U|$ stars 
is a nontrivial complete bipartite graph, implying the
statement of the lemma.
\vspace{0.2cm}

\noindent
{\bf Proof of Theorem \ref{t12}:}\,
Suppose $G=G(n,p)$ with $\frac{2}{n} \leq p \leq c$ and $c$ as in
Corollary \ref{c33}.
The required upper bound for
$\tau(G)$ follows from the well known fact that 
$\alpha(G)=\Theta(\frac{\log (np)}{p})$ whp (see \cite{BE} for a much
more precise result). We proceed with the proof of the lower bound.

The lower bound for $p=o(n^{-7/8})$ follows from the assertion of
Proposition \ref{p13}, proved in the next section. 
We thus may and will assume that, say,
$np \geq n^{0.1}$. Note that in this case $\log (np)=\Theta( \log
n)$.

By Corollary \ref{c33}, the probability that there exists a set 
$U$ of size $k \geq \frac{2\log n}{ap}$, for an appropriately
chosen absolute constant $a>0$,
so that $\tau'(G[U]) \leq 2|U|$ does not exceed
$$
{n \choose k} 2^{-apk^2} \leq 2^{k \log n-apk^2} \leq 
2^{k \log n-apk 2 \log n/(ap)}=2^{-k \log n} =n^{-k}.
$$
Note that to apply the Corollary, $k$ and $p$ should satisfy
$$
kp \geq C \log k.
$$
As $k \leq n$, and $k \geq \frac{2 \log n}{ap}$
it suffices to have 
$$
\frac{2 \log n}{ap} p=\frac{2 \log n}{a} \geq  C \log n
$$
and this holds by taking the constant $a$ as in Corollary
\ref{c33}, and by decreasing it to $2/C$ if it is larger (the
assertion of the Corollary clearly holds when $a$ is decreased).
Summing over all values  of $k\geq \frac{2\log n}{ap}$ we conclude
that whp there is no such set $U$. Suppose that's the case.

By Lemma \ref{l34}
there exists  a set of vertices 
$U \subset V$ so that if $G[U]$ denotes the induced subgraph of $G$
on $U$ then
$$
\tau(G)=n-|U|+\tau'(G[U]).
$$
Put $|U|=k$. If $k \leq \frac{2\log n}{ap}$ 
then 
$$
\tau(G)=n-|U|+\tau'(G[U]) \geq n-k \geq n-\frac{2\log n}{ap}
$$
providing the required estimate. For larger values of $k$, by 
the assumption above
$$
\tau(G)=n-|U|+\tau'(G[U]) \geq n-k+2k >n,
$$
providing the required bound (with room to spare). This completes the
proof.
\hfill$\Box$\medskip

\section{Concluding remarks}

We have shown that the conjecture of Erd\H{o}s that for
$G=G(n,0.5)$ the equality $\tau(G)=n-\alpha(G)$ holds whp is 
incorrect as stated.
The following slight variation of this conjecture seems
plausible.
\begin{conj}
\label{l41}
For the random graph $G=G(n,0.5)$,
$\tau(G)=n-\beta(G)+1$ whp.
\end{conj}
The more general conjecture of \cite{CP} that for any 
$p \leq 0.5$ and for $G=G(n,p)$,
$\tau(G)=n-(1+o(1))\alpha(G)$ whp may well be true. Although
we are not able to prove it, note that Theorem \ref{t12} proves a
similar, though 
weaker statement, namely, for all $p \leq c$ and for $G=G(n,p)$,
$\tau(G)=n-\Theta(\alpha(G))$ whp, where $c$ is an absolute
positive constant.

For $p < 0.5$ which is bounded away from $0.5$, it is easy to check
that for $G=G(n,p)$, $\beta(G) < \alpha(G)$ whp, and hence in this
range the upper bound $\tau(G) \leq n-\alpha(G)$ is typically
better than the upper bound $\tau(G) \leq n-\beta(G)+1$ (which is
much better for $p>0.5$, but we restrict our attention here to the
case $p \leq 0.5$).  For very sparse graphs, Proposition 
\ref{p13} determines precisely the typical value of $\tau(G)$.
Here is the simple proof.

\noindent
{\bf Proof of Proposition \ref{p13}:}\, By Lemma \ref{l34},
$$
\tau(G)=n-|U|+\tau'(G[U])
$$
for some set of vertices $U$ of the graph $G=G(n,p)$. However, when
$p=o(n^{-7/8})$ then, whp, $G$ contains no non-star complete
bipartite graphs besides $K_{2,2}=C_4$, and there are no two copies
of $C_4$ that share a vertex. Therefore, any connected component of 
the induced subgraph $G[U]$ on $U$ must be either an isolated vertex
or a cycle of length $4$, completing the proof. 
\hfill $\Box$ 
\vspace{0.2cm}

\noindent
For very sparse random graphs, namely, if $p=\Theta(1/n)$ and
$G=G(n,p)$, then the whole graph 
$G$ contains a connected component which is $C_4$
with probability that is bounded away from $0$ and $1$. If this is
the case, then the expression  $n-\mbox{max}(\gamma(H))$ provided
in Proposition \ref{p13} for $\tau(G)$ is strictly smaller than
$n-\alpha(G)$. Thus, for very sparse random graphs, it is  not the case
that $\tau(G)=n-\alpha(G)$ whp. Yet, it may well be the case that
for any fixed constant $p$ bounded away from $0$ and $0.5$,
$\tau(G)=n-\alpha(G)$ whp. At the moment we can neither prove  
nor disprove this statement, which remains open.
\vspace{0.2cm}

\noindent
{\bf Acknowledgment} I would like to thank Svante Janson for
helpful comments.

\end{document}